\documentclass{article}

\usepackage{amsmath,amssymb}
\usepackage{url}

\usepackage[all]{xy}

\usepackage[T1]{fontenc}

\usepackage{amssymb,latexsym, amsmath, amsxtra}

\usepackage{graphicx}
\usepackage{array}

\usepackage{enumerate}

\numberwithin{equation}{section}

\newcommand{\Z}{{\mathbb Z}}
\newcommand{\R}{{\mathbb R}}
\newcommand{\C}{{\mathbb C}}

\newcommand{\HH}{\mathcal H}

\newcommand{\GL}{{\rm GL}}

\newcommand{\SL}{{\rm SL}}

\newcommand{\Sp}{{\rm Sp}}
\newcommand{\GSp}{{\rm GSp}}

\newcommand{\SSp}{{\rm Sp}}

\newcommand{\binomial}[2]{\genfrac{(}{)}{0pt}{}{#1}{#2}}

\newcommand{\mat}[4]{{\setlength{\arraycolsep}{0.5mm}\left[
\begin{array}{cc}#1&#2\\#3&#4\end{array}\right]}}
\newcommand{\qed}{\hspace*{\fill}\rule{1ex}{1ex}}
\newcommand{\forget}[1]{}

\def\qdots{\mathinner{\mkern1mu\raise0pt\vbox{\kern7pt\hbox{.}}\mkern2mu
\raise3.4pt\hbox{.}\mkern2mu\raise7pt\hbox{.}\mkern1mu}}

\newenvironment{proof}{\vspace{0ex}\noindent{\it Proof.}\hspace{0.1em}}
	{\hfill\qed\vspace{2ex}}

\newtheorem{lemma}{Lemma.}[section]
\newtheorem{theorem}[lemma]{Theorem}

\begin{document}
\thispagestyle{empty}
\title{Characterizations of the Saito-Kurokawa lifting: a survey}

\author{
David W. Farmer
\\American Institute of Mathematics
\\farmer@aimath.org
\and
Ameya Pitale
\\University of Oklahoma
\\apitale@math.ou.edu
\and
Nathan C. Ryan
\\Bucknell University
\\nathan.ryan@bucknell.edu
\and
Ralf Schmidt
\\University of Oklahoma
\\rschmidt@math.ou.edu
}

\maketitle

\begin{abstract}

There are a variety of characterizations of Saito-Kurokawa lifts
from elliptic modular forms to Siegel modular forms of degree 2.  In
addition to giving a survey of known characterizations, we apply a recent result of Weissauer to provide a number of new
and simpler characterizations of Saito-Kurokawa lifts.  
\end{abstract}

\section{Introduction}

Consider a classical modular form $f$ of weight $k$ for the group $\SL_2(\Z)$
with Fourier expansion $\sum_{n\geq 0}a_nq^n$.  If $f$ is a Hecke
eigenform, one can define an $L$-function $L(s,f) = \sum_{n\geq 0}a_n
n^{-s}$ so that its completion satisfies a functional equation with
$s\mapsto k-s$.  In this setting, Deligne \cite{D1, D2} proved the
Ramanujan Conjecture:  namely, that the (suitably normalized) Satake parameters of $f$ are unimodular.  

In the setting of Siegel modular forms, the naive generalization of
the Ramanujan Conjecture above is false.  In particular, H.\ Saito and N.\ Kurokawa independently constructed and computed Hecke eigenforms that had Satake
parameters not on the unit circle; see \cite{Ku}.  It was later understood that these modular forms were in fact lifts from elliptic modular forms, and so the naive generalization of the Ramanujan Conjecture was adjusted to say that if a Siegel modular form is not a
Saito-Kurokawa lift then it has unimodular Satake parameters.
Weissauer proved this conjecture in \cite{W}.

In this paper, we exploit this recent result to provide a collection of new characterizations of Saito-Kurokawa lifts. The new characterizations we present
in Section~\ref{sec:SK} can be formulated as being determined by a single condition at a
single prime. The tools used in this result, aside from the Ramanujan Conjecture, are elementary.
Before we present this result, we describe a number of characterizations in the literature \cite{Br,E,EZ,H1,H2,O,PS}.  A common
feature of these earlier characterizations, unlike the new ones we
present, is that they are determined by an infinite number of
conditions.

The paper is organized as follows.  We begin by giving some background
on Siegel modular forms and their Hecke operators and fixing some
notation.  In the subsequent section we describe the Saito-Kurokawa
lift and the characterizations of its image that are found in the
literature.  We conclude by presenting our new characterizations.

\subsection*{Acknowledgements}
This work was supported by an AIM SQuaRE. Ameya Pitale and Ralf Schmidt are supported by National Science Foundation grant DMS  1100541.

\section{Siegel modular forms and Hecke operators}
Let the symplectic group of similitudes of genus $2$ be defined by
\begin{multline*}
 \GSp(4) := \{g \in \GL(4) : {}^{t}g J g = \lambda(g) J,
 \lambda(g) \in \GL(1) \} \\ \mbox{ where } J = \mat{}{I_2}{-I_2}{}.
\end{multline*}
Let $\SSp(4)$ be the subgroup with $\lambda(g)=1$. The group $\GSp^+(4,\R) := \{ g \in\GSp(4,\R) : \lambda(g) > 0 \}$ acts on the Siegel upper half space $\HH_2 := \{ Z \in M_2(\C) : {}^{t}Z = Z, {\rm Im}(Z) > 0\}$ by
\begin{equation}
 g \langle Z \rangle := (AZ+B)(CZ+D)^{-1}, \quad \text{where } g = \mat{A}{B}{C}{D} \in \GSp^+(4,\R), Z \in \HH_2.
\end{equation}
Let $S_k^{(2)}$ be the space of
holomorphic Siegel cusp forms of weight $k$, genus $2$ with respect to $\Gamma^{(2)} := \SSp(4,\Z)$. Then $F \in S_k^{(2)}$
satisfies
$$
 F(\gamma \langle Z \rangle) = \det(CZ+D)^{k} F(Z)
$$
for all $\gamma = \mat{A}{B}{C}{D} \in \Gamma^{(2)}$ and  $Z \in \HH_2$.

For 
$M \in \GSp^+(4,\R) \cap M_{4}(\Z)$, define the Hecke operator
$T_k(\Gamma^{(2)} M \Gamma^{(2)})$ on $S_k^{(2)}$ as in (1.3.3) of
\cite{An}.  
For a positive integer $m$,
we define the Hecke operator $T_k(m)$ by
\begin{equation}\label{hecke-op-m-defn}
T_k(m) := \sum\limits_{\lambda(M)=m} T_k(\Gamma^{(2)} M
\Gamma^{(2)}).
\end{equation}
Let us assume that $F \in S_k^{(2)}$ is a Hecke eigenfunction with $T_k(m) F = \mu_F(m)F$.  Using the generating function for $\mu_F(p^r)$ as in Theorem 1.3.2 of
\cite{An} we see that there are two complex numbers $\alpha_p, \beta_p$ such that

\begin{align}
 \mu_F(p)&= p^{k-3/2} \big(\alpha_p + \alpha_p^{-1} + \beta_p + \beta_p^{-1}\big),\label{p-eval-form}\\
 \mu_F(p^2)&=p^{2k-3}\big((\alpha_p+\alpha_p^{-1})^2+(\alpha_p+\alpha_p^{-1})(\beta_p+\beta_p^{-1})+(\beta_p+\beta_p^{-1})^2-2-\frac1p\big). \label{p^2-eval-form}
\end{align}
The numbers $\alpha_p, \beta_p$ are called the Satake parameters of $F$. 


\section{Saito-Kurokawa lifts and their characterization}

For even $k$, we have the following diagram of lifts.
\begin{equation}
\xymatrix{J_{k,1}^{cusp} \ar[r]^{Ma} &S_k^{(2)} \\ S_{k-\frac 12}^+(\Gamma_0(4)) \ar[u]^{EZ} & S_{2k-2}(\SL_2(\Z)) \ar[l]_{Sh} \ar[u]_{SK}}
\end{equation}
The four spaces in that diagram are:
\begin{itemize}
\item{}$S_{2k-2}(\SL_2(\Z))$ is the space of elliptic cusp forms of
weight $2k-2$ with respect to $\SL_2(\Z)$, 
\item{}$S_{k-\frac
  12}^+(\Gamma_0(4))$ is the space of holomorphic cusp forms of weight
$k-1/2$ with respect to $\Gamma_0(4)$ in the Kohnen plus space,
\item{}$J_{k,1}^{cusp}$ is the space of holomorphic cuspidal Jacobi forms of weight $k$ and
index~$1$, and
\item{}$S_k^{(2)}$ is the space of
holomorphic Siegel cusp forms of weight $k$, genus $2$ with respect to $\Gamma^{(2)} := \SSp(4,\Z).$ 
\end{itemize}
The map $Sh$ is the Shimura lift given explicitly by
Kohnen \cite{Ko}, the map $EZ$ is the lift from half integral weight modular
forms to Jacobi forms given by Eichler-Zagier \cite{EZ}, and the map
$Ma$ is the lift from Jacobi forms to Siegel modular forms due to
Maass \cite{M}. Finally, the map $SK$ is the Saito-Kurokawa lift given
as the composition of the three lifts described above. If $f \in
S_{2k-2}(\SL_2(\Z))$ then we write $F = SK(f)$ for the Saito-Kurokawa
lift of $f$. 

There are several different ways to characterize when $F \in
S_k^{(2)}$ is in the image of the Saito-Kurokawa lift. Here, we
present a few of these.   We begin by describing a number of characterizations related to our
new characterization found below.

\subsection{The Maass space}

Given a Siegel modular form $F$ of degree 2, we ask if the form is a
lift.  The first characterization that allows us to answer this
question is in terms of its Fourier coefficients.  Let
\begin{equation}\label{four-exp}
 F(\mat{\tau}{z}{z}{\tau'})
 = \sum\limits_{n,r,m}A(n,r,m)e^{2 \pi i (n \tau + rz + m \tau')},\qquad
 \mat{\tau}{z}{z}{\tau'} \in \HH_2,
\end{equation}
be the Fourier expansion of $F$. The \emph{Maass space} (also known as the Maass Spezialschar) consists of all $F$ in $S_k^{(2)}$ satisfying
\begin{equation}\label{maass space}
   A(n,r,m) = \sum\limits_{d | (n,r,m)} d^{k-1}A(\frac{nm}{d^2},
   \frac{r}{d},1) \quad\mbox{ for all } n,m,r \in \Z.
\end{equation}
By \cite{M}, \cite{An2}, \cite{Z}, $F$ is a Saito-Kurokawa lift if and only if $F$ lies in the Maass space. A streamlined proof of this fact is contained in \cite{EZ}.

We remark that in \cite{H1} a different Spezialschar is defined.  It
is conjectured to be equal to the Maass space in general and they are proved
to be the same in degree 2.

Both of these characterizations are fundamentally different than our
new characterizations in that they are global characterizations.  That is, they require knowing all the Fourier coefficients of $F$; in particular, they require
checking infinitely many conditions.

We remark that Ikeda \cite{I} has obtained a Saito-Kurokawa type lifting of elliptic cusp forms to Siegel modular forms of degree $2n$. The characterization of the image of the lifting via Maass conditions has been generalized to this higher rank situation in \cite{KK,Y}.

\subsection{The Maass $p$-space}

Let $p$ be a prime number. The \emph{Maass $p$-space} consists of all $F$ in $S_k^{(2)}$ satisfying
  \begin{multline}\label{p-maass space}
   A(np, r, m) + p^{k-1}A(\frac{n}{p}, \frac{r}{p},m) =\\ p^{k-1}A(n,
   \frac{r}{p}, \frac{m}{p}) + A(n,r,mp) \mbox{ for } n,r,m \in \Z.
  \end{multline}
Here, we understand that $A(\alpha,\beta,\gamma)=0$ if one of $\alpha,\beta,\gamma$ is not an integer. If $F$ is a Saito-Kurokawa lift, then it lies in the Maass $p$-space for every prime $p$; this follows by substituting (\ref{maass space}) into (\ref{p-maass space}). In \cite{PS}, it was shown that if $F$ lies in the Maass $p$-space for almost all $p$ then $F$ is a Saito-Kurokawa lift.

In \cite{H2} the condition that $F$ is in the Maass $p$-space is ``$F$
satisfies $p$-Hecke duality'' and in that paper the following is proved: suppose $R$ is any set of prime numbers with Dirichlet density smaller than $1/8$. Then $F$ is a Saito-Kurokawa lift if $F$ satisfies $p$-Hecke duality for any $p$ outside $R$.  This can be thought of as an improvement on the result described above.

These two related characterizations, while, in a sense, local, also
require checking infinitely many conditions: one condition for
each prime in some infinite set.  

Also in \cite{H2}, the following result is proved.  For a fixed
positive even integer~$k$ there exists a constant $c(k)$ depending
only on the weight $k$ such that a cusp form $F\in S^{(2)}_k$ is a Saito-Kurokawa
lift if and only if $F$ satisfies $p$-Hecke duality for all prime
numbers $p\leq c(k)$.  

In \cite{RS2}, it is proved, under the assumption of the Ramanujan
Conjecture, that a Siegel modular form lies in the Maass space  if and
only if there exists a prime $p$ such that $F$ lies in the Maass
$p$-space.  This is an improvement of the previously described
result.  We recall here that Weissauer \cite{W} has recently
proved the Ramanujan Conjecture in this context and so this result is
now unconditional.

Both these results are local and require checking a condition for
finitely many~$p$.  Nevertheless, they still each require checking
infinitely many conditions as one iterates through the coefficients
in \eqref{p-maass space}.  Our new characterization of Hecke
eigenforms in the Maass space is an
improvement on this result as it requires checking a single condition for a
single prime.  We identify the
prime for which it should be checked,
and it turns out the prime is independent of the weight~$k$.

\subsection{Other characterizations}

Here we summarize a number of characterizations in the literature that
are not as closely related to our characterizations found below.  Each
of them requires global information.  In the first case it
is necessary to know all the coefficients of $F$; in the second case
it is necessary to be able to compute the pole of an $L$-function and so
one needs to know all the Hecke eigenvalues of $F$; and in the third,
one needs to check infinitely many conditions, namely that each Hecke
eigenvalue is positive.  The first characterization is not about
eigenforms while the other two are.

\subsubsection{A differential operator}

In \cite{H1} a characterization of forms in the Maass space is given
in terms of Taylor coefficients and a differential operator
$\mathbb{D}_{k,2\nu}:M_k^{(2)}\to {\rm Sym^2}(M_{k+2\nu})$ where
$v\in\Z_{\geq 0}$.  The characterization is:  a Siegel modular form $F$ of
degree 2 is in the Maass space if and only if $\mathbb{D}_{k,2\nu}F$ is in a
``diagonal'' subspace of ${\rm Sym^2}(M_{k+2\nu})$.  

\subsubsection{The $L$-function of a Saito-Kurokawa lift}

 Suppose $F$ is a Hecke eigenform.  Then \cite{E,O} show that $F$ is a
 Saito-Kurokawa lift if and only if the spin $L$-function of $F$ has
 poles.  This result has been generalized in the
 following way.  The Gritsenko
  lift is an analogue of the Saito-Kurokawa lift but yields a
  Siegel paramodular form of degree 2 and paramodular level $N$.  
In \cite{RS2} it is proved that $F$ is a Gritsenko lift if and only if
its spin $L$-function has poles.

\subsubsection{The Hecke eigenvalues of a Saito-Kurokawa lift}

Suppose $F$ is a Hecke eigenform.  Up to now the characterizations
have been in terms of Fourier coefficients, poles of $L$-functions, and
Taylor series coefficients.
The final characterization, due to Breulmann \cite{Br}, is as follows: $F$ is a
Saito-Kurokawa lift if and only if $\mu_F(m) > 0$ for all positive
integers $m$.  This characterization is very succinct and appealing
and has been studied further.  In particular, the following refinement
was obtained in Corollary 3.2 of \cite{PS1}: If $F$ is not a
Saito-Kurokawa lift, then there exists an infinite set $S_F$ of prime
numbers $p$ such that if $p\in S_F$, then there exist infinitely many
$r$ such that $\mu_F(p^r)>0$, and inifinitely many $r$ such that
$\mu_F(p^r)<0$.




\section{New characterizations of Saito-Kurokawa lifts}\label{sec:SK}

The characterizations of Saito-Kurokawa lifts we have described above are, in a sense, not
very effective.  
Given a Siegel modular form $F$, it is difficult to effectively determine if it is
a Saito-Kurokawa lift.

Our new
characterizations are local and can be used to develop a test that
reduces the problem of determining if a form is a Saito-Kurokawa lift
to checking a single local
condition (in fact, the prime at which the condition should be checked
is identified).  In this sense, we feel that
our characterizations are optimal.

The key point are the following estimates for the Satake parameters $\alpha_p, \beta_p$ of a cuspidal Siegel eigenform $F \in S_k^{(2)}$.

{\it Ramanujan estimate \cite{W}}: If $F$ is not a Saito-Kurokawa lift, then for every prime $p$, we have
\begin{equation}\label{ram-est}
|\alpha_p| = |\beta_p| = 1.
\end{equation}

{\it Saito-Kurokawa estimate}: In case $F$ is a Saito-Kurokawa lift
then we have for every prime $p$, after possibly changing notations,


\begin{equation}\label{SK-satake-estimates}
\alpha_p = p^{1/2}, \qquad |\beta_p| = 1.
\end{equation}
We get the above condition from the following relation between $L$-functions (see \cite{EZ}): 
\begin{equation}
 L(s, SK(f)) = L(s,f) \zeta(s+1/2) \zeta(s-1/2)
\end{equation}
The key point to note is that, if $F \in S_k^{(2)}$ is a Hecke
eigenform, then either $F$ satisfies the Ramanujan estimates
(\ref{ram-est}) or it is a Saito-Kurokawa lift. 

The following theorem implies that a Hecke eigenform $F\in S_k^{(2)}$ is a
Saito-Kurokawa lift if and only if $\mu_F(37)>4\cdot 37^{k-3/2}$.
However, we state it in a more general way. 

\begin{theorem}\label{eval-thm}
 Let $F \in S_k^{(2)}$ be a Hecke eigenform with eigenvalues $\mu_F(m)$ for any positive integer $m$. The following statements are equivalent.
 \begin{enumerate}
  \item $F$ is a Saito-Kurokawa lift.
  \item There exists a prime $p$ such that $\mu_F(p) > 4 p^{k-3/2}$.
  \item For every prime $p \geq 37$ we have $\mu_F(p) > 4 p^{k-3/2}$.
  \item There exists a prime $p$ such that $\mu_F(p^2) > 10 p^{2k-3}$.
  \item For every prime $p \geq 17$ we have $\mu_F(p^2) > 10 p^{2k-3}$.
  \item For every prime $p$ we have
   \begin{equation}\label{mu-condition}
    \mu_F(p^2) = \mu_F(p)^2 - (p^{k-1} + p^{k-2}) \mu_F(p) + p^{2k-2}.
   \end{equation}
  \item There exists a prime $p$ such that
   \begin{equation}\label{mu-conditionone}
    \mu_F(p^2) = \mu_F(p)^2 - (p^{k-1} + p^{k-2}) \mu_F(p) + p^{2k-2}.
   \end{equation}
\end{enumerate}
\end{theorem}
\begin{proof}
If $F$ is not a Saito-Kurokawa lift then, by (\ref{p-eval-form}), (\ref{p^2-eval-form}) and (\ref{ram-est}), we have $|\mu_F(p)| \leq 4 p^{k-3/2}$ and $|\mu_F(p^2)| \leq 10 p^{2k-3}$ for every prime $p$. If we substitute $\alpha_p = \sqrt{p}$ in  (\ref{p-eval-form}) and (\ref{p^2-eval-form}), then we get $\mu_F(p) > 4 p^{k-3/2}$ for $p \geq 37$ and $\mu_F(p^2) > 10 p^{2k-3}$ for $p \geq 17$. This proves the equivalence of $i)$ through $v)$.

We will now show that $i) \Rightarrow vi) \Rightarrow vii) \Rightarrow i)$. Let us assume that $F$ is a Saito-Kurokawa lift. Then, for any prime $p$, if we substitute $\alpha_p = \sqrt{p}$ in (\ref{p-eval-form}) and solve for $\beta_p + \beta_p^{-1}$ and substitute the result in (\ref{p^2-eval-form}), we get (\ref{mu-condition}). This gives $i) \Rightarrow vi)$. The implication $vi) \Rightarrow vii)$ is trivial. Now, suppose for some prime $p$ the condition (\ref{mu-conditionone}) is satisfied. Then
\begin{multline}
  \mu_F(p)^2 - \mu_F(p^2) - p^{2k-4} =\\ 2 p^{2k-3} +
  p^{k-3/2}(p^{1/2}+p^{-1/2}) \big(\mu_F(p) -
  p^{k-3/2}(p^{1/2}+p^{-1/2})\big).
\end{multline}
On the other hand, eliminating $\beta_p+\beta_p^{-1}$ from \eqref{p-eval-form} and \eqref{p^2-eval-form}, we get
\begin{equation}
 \mu_F(p)^2 - \mu_F(p^2) - p^{2k-4} = 2 p^{2k-3} + p^{k-3/2}(\alpha_p + \alpha_p^{-1}) \big(\mu_F(p) - p^{k-3/2}(\alpha_p + \alpha_p^{-1})\big).
\end{equation}
Hence, $Z_1=p^{k-3/2}(p^{1/2}+p^{-1/2})$ and $Z_2=p^{k-3/2}(\alpha_p + \alpha_p^{-1})$ are the roots of the same quadratic equation $Z^2-\mu_F(p) Z + t = 0$, where the value of $t$ is obtained from the above equations. If $Z_1 = Z_2$ then $\alpha_p = p^{\pm 1/2}$ and we obtain $i)$. If $Z_1 \neq Z_2$, then, from $Z_1 + Z_2 = \mu_F(p)$, we see that $\beta_p = p^{\pm 1/2}$ and we obtain $i)$. This completes the proof of the theorem.
\end{proof}

Finally, we give a characterization of non-Saito-Kurokawa lifts in
terms of the growth of the sequence of numbers $|\mu_F(p^r)|$ as
$r$ goes to infinity. For a fixed prime $p$, the following condition
was considered in \cite{P}: For all $\epsilon>0$ there exists a
positive constant $C_\epsilon>0$, depending only on $p$ and $\epsilon$,
such that
\begin{equation}\label{ameya-cond}
 |\mu_F(p^r)|\leq C_\varepsilon p^{r(k-3/2+\varepsilon)}\qquad\text{for all }r\geq0.
\end{equation}
It was shown that this condition is equivalent to the Ramanujan
estimate \eqref{ram-est} for the Satake parameters at $p$. 
For the case of $\Sp(4,\Z)$, using the Ramanujan estimate \eqref{ram-est}, condition~\eqref{ameya-cond} can be written in the more
precise (but equivalent) form
\begin{align}
 |\mu_F(p^r)|\leq \mathstrut &  \left(\binomial{r+3}{3} + p^{-1} \binomial{r+1}{3} \right) p^{r(k-3/2)}
\\
\leq \mathstrut & \frac{3}{2}\, \binomial{r+3}{3} \, p^{r(k-3/2)} \qquad\text{for all }r\geq0.
\label{ameya-cond2}
\end{align}
Since \eqref{ram-est} for \emph{one} $p$ implies \eqref{ram-est} for \emph{all} $p$, we obtain the following result.

\begin{theorem}\label{ameya-thm}
 Let $F \in S_k^{(2)}$ be a Hecke eigenform with eigenvalues $\mu_F(m)$ for any positive integer $m$. The following statements are equivalent.
 \begin{enumerate}
  \item $F$ is not a Saito-Kurokawa lift.
  \item There exists a prime $p$ such that \eqref{ameya-cond2} holds.
  \item For every prime $p$ \eqref{ameya-cond2} holds.
 \end{enumerate}
\end{theorem}








\end{document}